\documentclass[a4paper,11pt]{article}
\usepackage{amsfonts}
\usepackage[T1]{fontenc}
\usepackage{microtype}
\usepackage{makeidx}
\usepackage{setspace}
\usepackage{authblk}
\usepackage{graphicx}
\usepackage{xcolor}
\usepackage[all]{xy}
\usepackage{amsmath}
\usepackage{amssymb}
\usepackage{booktabs}
\usepackage{braket}
\usepackage{array}
\usepackage{tabularx}
\usepackage{multirow}
\usepackage{indentfirst}
\usepackage{cite}
\usepackage{stmaryrd}
\usepackage{faktor}
\usepackage{titling}
\usepackage{tikz}
\usepackage{dsfont}
\usepackage{url}
\usepackage{hyperref}

\usepackage{tkz-graph}

\usetikzlibrary{matrix,arrows,decorations.pathmorphing}
    \renewcommand{\leq}{\leqslant}
    
\usepackage{amsthm}
\theoremstyle{plain}
\newtheorem{thm}{Theorem}[section]
\newtheorem*{thm*}{Theorem}

\newtheorem{prop}[thm]{Proposition}
\newtheorem{cor}[thm]{Corollary}

\theoremstyle{remark}

\theoremstyle{definition}
\newtheorem{ex}[thm]{Example}

\DeclareMathOperator{\End}{\mathrm{End}}

\DeclareMathOperator{\N}{\mathbb{N}}

\DeclareMathOperator{\R}{\mathbb{R}}

\DeclareMathOperator{\var}{\vartriangleleft}

\title{\bf{Reflections and quotients in Coxeter groups}}
\author{}
\author{Paolo Sentinelli\thanks{ Dipartimento di Matematica, Politecnico di Milano, Milan, Italy. \\ \href{mailto:paolosentinelli@gmail.com}{paolosentinelli@gmail.com}}}
\date{}
\begin{document}

\maketitle

\vspace{-4em}

\begin{abstract}
We present a formula relating the set of left descents of an element of a Coxeter group with the sets of left descents of its projections on maximal quotients indexed by simple right descents. This formula is an instance of a general result involving quotients and descents.
\end{abstract}

$\,$

\section{Introduction}
\vspace{1em}
A reflection $t$ in a Coxeter system $(W,S)$ is a conjugate of a generator, and $t$ is a left descent of an element $w\in W$ if $\ell(tw)<\ell(w)$, where $\ell$ is the length function. Denoting by $T_L(w)$ the set of left descents of $w$, it is well known that the right weak order $\leq_R$ on $W$ is isomorphic to the set $\{T_L(w):w\in W\}$ ordered by inclusion.   In this note we prove the following formula:
\vspace{1em}
\begin{equation*}
 T_L(v)=T_L(u)\, \cup \bigcup_{s\in D_R(u^{-1}v)}T_L(v^{S\setminus\{s\}}),   
\end{equation*} 

\vspace{1em}
\noindent for all $u,v \in W$ such that $u\leq_R v$. Here $D_R(w)$ is the set of right simple descent (right descents of length one) of $w$, and $w^J$ is the element of minimal length in the coset $wW_J$, for any $J\subseteq S$. This formula is special case of a general result we prove in Theorem \ref{teorema riflessioni}. We then deduce some corollaries; for example, our result implies that the set $\{w^{S\setminus\{s\}}: s\in D_R(w)\}$ admits a join and it is precisely $w$ (Corollary \ref{cor join}), i.e.
\vspace{0.2em}
\begin{equation*}
    w=\bigvee\limits_{s\in D_R(w)}w^{S\setminus \{s\}},
\end{equation*} for all $w\in W$. We are regarding the poset $(W,\leq_R)$ as a complete meet-semilattice, for which the joins may not exist.

\vspace{1em}
We now fix notation and recall some definitions useful for the rest of the paper. We refer to the book of Bj\"{o}rner and Brenti \cite{BB} for Coxeter groups.

Let $(W,S)$ be a Coxeter system, i.e. a presentation of the group $W$ is given by a set $S$ of involutive generators and relations encoded by a \emph{Coxeter matrix}
$m:S\times S \rightarrow \{1,2,...,\infty\}$. A Coxeter matrix over $S$ is a symmetric matrix which
satisfies the following conditions for all $s,t\in S$:
\begin{enumerate}
  \item $m(s,t)=1$ if and only if $s=t$;
  \item $m(s,t)\in \{2,3,...,\infty\}$ if $s\neq t$.
\end{enumerate}  The presentation
$(W,S)$ of the group $W$ is then the following:
$$\left\{
  \begin{array}{ll}
    \mathrm{generators}: & \hbox{$S$;} \\
    \mathrm{relations}: & \hbox{$(st)^{m(s,t)}=e$,}
  \end{array}
\right.$$ for all $s,t\in S$, where $e$ denotes the identity in $W$. The Coxeter matrix $m$ attains the value $\infty$ at $(s,t)$ to indicate that there is no relation
between the generators $s$ and $t$.
The class of words expressing an element of $W$ contains words of minimal length; the \emph{length function} $\ell: W \rightarrow
\mathbb{N}$ assigns to an element $w\in W$ such minimal length. The identity $e$ is represented by the empty word and then $\ell(e)=0$. A \emph{reduced word} or \emph{reduced expression} for an element $w\in W$ is a word of minimal length representing $w$.  The set of \emph{reflections} of $(W,S)$ is defined by $T:=\{wsw^{-1}:w\in W, \, s\in S\}$. For $J\subseteq S$ and $v\in W$, we let
\begin{eqnarray*} W^J&:=&\Set{w\in W:\ell(w)<\ell(ws)~\forall~s\in J},
\\ D_R(v)&:=&\Set{s\in S:\ell(vs)<\ell(v)}.
\end{eqnarray*}
With $W_J$ we denote the subgroup of $W$ generated by $J\subseteq S$; in particular, $W_S=W$ and $W_\varnothing =
\Set{e}$. 
We let $\leqslant_R$  and $\leqslant$ be the \emph{right weak order}  and the \emph{Bruhat order}  on $W$, respectively. The covering relations of the right weak order are characterized as follows: 
$u \var_R v$ if and only if $\ell(u)<\ell(v)$ and $u^{-1}v \in S$. 
For $w\in W$, define $$T_L(w):=\{t\in T: \ell(tw)<\ell(w)\},$$
$$T_R(w):=\{t\in T: \ell(wt)<\ell(w)\}.$$ The right weak order has the following property (see \cite[Proposition 3.1.3]{BB}):
$$u \leq_R v \, \, \Leftrightarrow \, \, T_L(u)\subseteq T_L(v).$$
The covering relations of the Bruhat order are characterized as follows:
$u \var v$ if and only if $\ell(u)=\ell(v)-1$ and $u^{-1}v \in T$.
The posets $(W,\leqslant_R)$ and $(W,\leqslant)$ are graded with rank function $\ell$ and $(W,\leqslant_R) \hookrightarrow (W,\leqslant)$. 
For $J\subseteq S$, each element $w\in W$ factorizes uniquely as
$w=w^Jw_J$, where $w^J\in W^J$, $w_J\in W_J$ and $\ell(w)=\ell(w_J)+\ell(w^J)$; see~\cite[Proposition~2.4.4]{BB}. We consider the idempotent function $P^J:W
\rightarrow W$ defined by
\begin{equation*} P^J(w)=w^J,
\end{equation*} for all $w\in W$. This function is order preserving for the Bruhat order
(see \cite[Proposition~2.5.1]{BB}). Moreover $I\subseteq J \subseteq S$ implies $P^J\circ P^I = P^J$ (see \cite[Lemma 5.4]{deodhar}). We let $$P^{(s)}:=P^{S\setminus \{s\}},$$ for all $s\in S$.
\noindent For the following proposition see  \cite[Corollary 2.6.2]{BB}. \vspace{1em}
\begin{prop} \label{deodhar}
    Let $u,v\in W$. Then 
    $$u\leq v \, \, \Leftrightarrow \, \, P^{(s)}(u) \leq P^{(s)}(v) \, \,\forall \, s \in D_R(u).$$
\end{prop} 

\vspace{1em}
Now we recall some results about the standard geometric representation of a finitely generated Coxeter group.
Let $(W,S)$ be a Coxeter system with $|S|<\infty$ and Coxeter matrix $m$. Let $V$ be the free $\R$-vector space on the set $\{\alpha_s:s \in S\}$ and $\sigma_s \in \End(V)$ defined by $\sigma_s(v)=v-2(\alpha_s|v)\alpha_s$, where the bilinear form $(\cdot|\cdot)$ is defined by setting $(\alpha_s|\alpha_t):=-\cos\left(\frac{\pi}{m(s,t)}\right)$, for all $s,t\in S$; then the assignment $w \mapsto \sigma_w:=\sigma_{s_1}\cdots \sigma_{s_k}$, where $s_1\cdots s_k$ is any reduced word of $w\in W$,  provides a faithful representation
$W \rightarrow \mathrm{GL}(V)$ (see, e.g. \cite[Ch.~4]{BB}). Then we let $w(v):=\sigma_w(v)$, for all $v\in V$, $w\in W$.
Let $\Phi:=\{w(\alpha_s): w\in W, \, s \in S\}$,
$V^+:=\left\{\sum_{s\in S}a_s\alpha_s: \, a_s \geqslant 0 \, \forall \, s\in S\right\} \setminus \{0\}$ and $V^-:=\left\{\sum_{s\in S}a_s\alpha_s: \, a_s \leqslant 0 \, \forall \, s\in S\right\} \setminus \{0\}$; then it is well known that
$\Phi=\Phi^+ \uplus \Phi^-=\{\alpha_t : t \in T\} \uplus \{-\alpha_t : t \in T\}$, $\Phi^+ \subseteq V^+$ and $\Phi^- \subseteq V^-$.
It holds that $t(v)=v-2(\alpha_t|v)\alpha_t$ and $w(\alpha_t) \in \Phi^-$ if and only if $wt<w$, for 
all $w\in W$, $t\in T$ (see, e.g. \cite[Sec.~4.4]{BB}). 
Moreover, if $s\in S$ and $t\in T\setminus \{s\}$, then
$s(\alpha_t)=\alpha_{sts}$.

\section{Reflections and quotients}\label{sezione Coxeter}
\vspace{1em}
In this section we prove the main result of this paper, i.e. the formula mentioned in the introduction, as a corollary of the following theorem.

\begin{thm} \label{teorema riflessioni} Let $(W,S)$ be a Coxeter system and $u,v\in W$ such that   
$u\leq_R v$. If $E\subseteq \mathcal{P}(S)$ is such that $\bigcap_{J\in E}J=S\setminus D_R(u^{-1}v)$, then 
    $$T_L(v)=T_L(u)\, \cup \bigcup_{J\in E}T_L(v^J).$$
\end{thm}
\begin{proof}  It is known that 
\begin{equation} \label{eq0}
    T_L(xy)=T_L(x)+xT_L(y)x^{-1},
\end{equation}
for all $x,y \in W$, where $+$ is the symmetric difference. 
Hence, for any $v\in W$, we have that $\varnothing = T_L(v^{-1}v) =T_L(v^{-1})+v^{-1}T_L(v)v$; this implies $T_R(v)=T_L(v^{-1})=v^{-1}T_L(v)v$, for all $v\in W$. Consider $J\subseteq S$. Then 
\begin{equation} \label{eq1}
    (v^J)^{-1}T_L(v^J)v^J \cap W_J = T_R(v^J)\cap W_J= \varnothing.
\end{equation} Notice that, if $\ell(xy)=\ell(x)+\ell(y)$, the symmetric difference in \eqref{eq0} is a disjoint union.
Therefore we have that
\begin{equation} \label{eq2}
    T_L(v) = T_L(v^J) \uplus v^JT_L(v_J)(v^J)^{-1} ,
\end{equation}
for all $v\in W$, $J\subseteq S$. Equalities \eqref{eq1} and \eqref{eq2} imply 
\begin{eqnarray} \label{eq3} \begin{split}
    T_L(v^J)&=T_L(v)\setminus v^JT_L(v_J)(v^J)^{-1} \\
    &= T_L(v)\setminus v^JW_J(v^J)^{-1} = T_L(v)\setminus vW_Jv^{-1}, \end{split}
\end{eqnarray}
If $u\leq_R v$ and $w:=u^{-1}v$, we have that
$T_L(v)=T_L(u)\uplus uT_L(w)u^{-1}$.
Moreover, by \eqref{eq3},
\begin{eqnarray} \label{eq5} \begin{split}
 uT_L(w^J)u^{-1} &= u(T_L(w)\setminus wW_Jw^{-1})u^{-1} \\
 &= uT_L(w)u^{-1}\setminus vW_Jv^{-1} \\ & \subseteq T_L(v ) \setminus vW_Jv^{-1} 
 = T_L(v^J). \end{split}
\end{eqnarray}
Therefore, by \eqref{eq3} and \eqref{eq5}, if we let $K:=\bigcap_{J\in E}J$, we obtain
\begin{eqnarray*}
    \bigcup\limits_{J\in E} T_L(v^J) & \supseteq & \bigcup\limits_{J\in E} uT_L(w^J)u^{-1}\\
    &=& u\left(\bigcup\limits_{J\in E} (T_L(w)\setminus wW_Jw^{-1})\right)u^{-1} \\
    &=& u\left(T_L(w)\setminus wW_Kw^{-1}\right)u^{-1} \\
    &=& u\left(T_L(w)\setminus wW_{S\setminus D_R(w)}w^{-1}\right)u^{-1} \\
    &=& uT_L(w^{S\setminus D_R(w)})u^{-1} = uT_L(w)u^{-1}.
\end{eqnarray*}
This implies that $$T_L(u)\, \cup \bigcup_{J\in E}T_L(v^J) \, \supseteq \, T_L(u) \cup uT_L(w)u^{-1} \, = \, T_L(v).$$
On the other hand, since $u\leq_R v$ and $v^J\leq_R v$ for all $J\in E$,   we also have that $T_L(v) \supseteq T_L(u)\, \cup \, \bigcup_{J\in E}T_L(v^J)$ and this concludes the proof.
\end{proof}
By the fact that $\bigcap_{s\in D_R(u^{-1}v)}(S\setminus \{s\})=S\setminus D_R(u^{-1}v)$, we obtain the following corollary, which is the finest instance of the statement of Theorem \ref{teorema riflessioni}.
\begin{cor} \label{corollario teorema riflessioni} Let $(W,S)$ be a Coxeter system and $u,v\in W$ such that   
$u\leq_R v$. Then 
    $$T_L(v)=T_L(u)\, \cup \bigcup_{s\in D_R(u^{-1}v)}T_L(v^{S\setminus\{s\}}).$$
\end{cor}

In the following example we consider a Coxeter system of type $B_4$. The set of generators is $S=\{s_0,s_1,s_2,s_3\}$; the Coxeter matrix takes the values $m(s_0,s_1)=4$ and $m(s_1,s_2)=m(s_2,s_3)=3$.
\begin{ex}
In type $B_4$, let $u=s_2s_3s_2$ and $v=s_2s_3s_2s_1s_0s_2s_3$.
Then $u^{-1}v=s_1s_0s_2s_3$, $D_R(v)=\{s_0,s_2,s_3\}$ and $D_R(u^{-1}v)=\{s_0,s_3\}$. We have that
\begin{itemize}
    \item $T_L(u)=\{s_2,s_3,s_2s_3s_2\}$, 
    \item $T_L(v^{S\setminus \{s_3\}})=T_L(s_1s_2s_3)=\{s_1,s_1s_2s_1, s_1s_2s_3s_2s_1\}$,
    \item $T_L(v^{S\setminus \{s_0\}})=T_L(s_3s_2s_1s_0)=\{s_3,s_2s_3s_2,s_1s_2s_3s_2s_1,s_3s_2s_1s_0s_1s_2s_3\}$,
    \item $T_L(v)=\{s_1,s_2,s_3,s_1s_2s_1, s_2s_3s_2,  s_1s_2s_3s_2s_1, s_3s_2s_1s_0s_1s_2s_3\}$.
\end{itemize}
\end{ex}

\vspace{1em}
An element $w\in W$ is uniquely determined by its descent set; hence, by Theorem \ref{teorema riflessioni} for $u=e$, we have that $w$ is uniquely determined by the set $\{w^{S\setminus \{s\}}: s\in D_R(w)\}$. This fact can be also deduced from Proposition \ref{deodhar}. An additional information provided by Theorem \ref{teorema riflessioni} is the following one, which concerns the join operation in the right weak order. 

\vspace{1em}
\begin{cor} \label{cor join}
    Let $w\in W$. Then $w= \bigvee\limits_{s\in D_R(w)}w^{S\setminus \{s\}}$.
\end{cor}
\begin{proof}
Clearly $T_L(w^J) \subseteq T_L(w)$, for all $J\subseteq S$, or, equivalently, $w^J\leq_R w$, for all $J\subseteq S$. Hence the set $\{w^{S\setminus \{s\}}: s\in D_R(w)\}$ is bounded and then it admits a join $w'$ (see \cite[Theorem 1.5]{dyer}). Moreover, by \cite[Proposition 2.14]{hohlweg} we have that $T_L(w')=T_L(w)$ and this concludes the proof.
\end{proof} \vspace{1em}
By repeated use of Corollary \ref{corollario teorema riflessioni}, we obtain the following result.
\vspace{1em}
\begin{cor}
    Let $s_1s_2\cdots s_k$ be a reduced expression for $w\in W$. Then 
    \begin{equation*}
       T_L(w) = \bigcup\limits_{i=0}^{k-1}T_L((s_1\cdots s_{k-i})^{S\setminus \{s_{k-i}\}}).
    \end{equation*} 
\end{cor}

\vspace{1em}
We now deduce from Theorem \ref{teorema riflessioni} the statement of \cite[Lemma 4.2]{io2}. Such result has inspired our main formula and it is related to the shellability of Coxeter complexes.
\begin{cor} \label{corollario riflessioni}
    Let $u,v,w\in W$ be such that $u\leq_R v$
    and $u\leq_R w$. If $P^{(s)}(v)=P^{(s)}(w)$ for all $s\in D_R(u^{-1}v)$, then $v\leq_R w$.
\end{cor}
\begin{proof} Since $u\leq_R w$, we have that $T_L(u)\subseteq T_L(w)$. Moreover $T_L(v^{S\setminus \{s\}})=T_L(w^{S\setminus \{s\}})\subseteq T_L(w)$ for all $s\in D_R(u^{-1}v)$. Hence by Corollary \ref{corollario teorema riflessioni} we obtain $T_L(v) \subseteq T_L(w)$, which is equivalent to $v\leq_R w$.
\end{proof}
\vspace{0.3em}
For $n\in \N$, let $B_n$ be the power set $\mathcal{P}(\{1,\ldots,n\})$ ordered by inclusion and, for $w\in W$ and $K\subseteq S\setminus D_R(w)$, 
$$\mathcal{T}^K_w:=\left\{T_L(w^{K\cup J}) : J \subseteq D_R(w)\right\} \subseteq \mathcal{P}(T),$$ 

\noindent ordered by inclusion.
\vspace{0.3em}
\begin{prop} \label{prop minimale}
Let $w\in W$, $K\subseteq S\setminus D_R(w)$ and $n:=|D_R(w)|$. Then, as posets, $\mathcal{T}^K_w \simeq B_n$.
\end{prop}

\begin{proof} 
If $w=e$ the result is trivial. So let $w\neq e$ and $H:=\{s\in S: s\leq w\}$. We consider the standard geometric representation of the finitely generated Coxeter group $W_H$. Then $w\in W_H$ and
assume $T_L(w^{K\cup I}) = T_L(w^{K\cup J})$, for some $I,J\subseteq D_R(w)$.  Consider the factorization $w=w^{K \cup J}w_{K \cup J}$.
We have that $s \in J\subseteq D_R(w)$ if and only if $s\in D_R(w_{K \cup J})$, i.e. if and only if  $w_{K \cup J}(\alpha_s) \in \Phi^-$. Let $t_s:= wsw^{-1}$; then $\alpha_{t_s}=-w(\alpha_s)$. We have that 
$$        (w^{K \cup J})^{-1}(\alpha_{t_s})= -(w^{K \cup J})^{-1}w(\alpha_s) 
        =-w_{K \cup J}(\alpha_s).
    $$ Hence
    \begin{eqnarray*}
        s \in J &\Leftrightarrow & s\in D_R(w_{K \cup J}) \\
        &\Leftrightarrow & (w^{K \cup J})^{-1}(\alpha_{t_s}) \in \Phi^+ \\
        &\Leftrightarrow & t_s \not \in T_L(w^{K \cup J})\\
         &\Leftrightarrow & t_s \not \in T_L(w^{K \cup I}) \\
         &\Leftrightarrow & s\in D_R(w_{K \cup I}) \Leftrightarrow  s \in I.
    \end{eqnarray*}
Therefore $I=J$. We have proved that $|\mathcal{T}_w|=|B_n|=2^n$.
    This also proves that $T_L(w^{K \cup I}) \subseteq T_L(w^{K \cup J})$ implies $J \subseteq I$. On the other hand, if
 $I \subseteq J$, then $K \cup I \subseteq K \cup J$ and we have that $w^{K \cup J}=(P^{K \cup J}\circ P^{K \cup I})(w) \leq_R w^{K \cup I}$, i.e. $T_L(w^{K \cup J}) \subseteq T_L(w^{K \cup I})$.
Then the poset $\mathcal{T}_w$ is isomorphic to the dual of $B_n$ and the result follows since $B_n$ is self--dual.
\end{proof}
Proposition \ref{prop minimale} implies that the atoms of the poset $\mathcal{T}^{S\setminus D_R(w)}_w$ are the sets $T_L(w^{S\setminus \{s\}})$, with $s\in D_R(w)$. The proof of this proposition also shows that if $s\in D_R(w)$ then
$$wsw^{-1} \in T_L(w^{S\setminus \{s\}}) \setminus \bigcup\limits_{r\in D_R(w)\setminus \{s\}} T_L(w^{S\setminus \{r\}}),$$
i.e. the union in Corollary \ref{corollario riflessioni} is minimal in respect to its indexing set.

\section{Acknowledgements}

I thank the anonymous referee for having generalized the original main result (now Corollary \ref{corollario teorema riflessioni}) and proposing a more elegant proof.
I also thank Christophe Hohlweg for some useful conversations and for having pointed out the result of Corollary \ref{cor join}.

\end{document}